\documentclass{article}                     
\usepackage[margin=2cm]{geometry}
\usepackage{graphicx}

\usepackage[utf8]{inputenc}
\usepackage{amsfonts,amssymb,amsmath,amsthm}
\usepackage{color}

\newtheorem{theorem}{Theorem}

\newtheorem{remark}{Remark}
\graphicspath{{figures/}}

\usepackage{color}
\newcommand{\bbN}{{\mathbb N}}   
\def\bbZ{\mathbb{Z}}   
\def\Cessum{{\lower 1.5pt \hbox{\large C}}{-}\!\!\!\sum}
\def\CHI{\hbox{\raise .5ex \hbox{$\chi$}}}
\newcommand{\GLsp}{{\boldsymbol G\kern-0.1em\boldsymbol L}}

\newcommand{\inner}[2]{\langle #1,#2\rangle}   
\def\Msp{{\boldsymbol M}}   
\def\mv1{\Msp_v^1}   

\def\sabs{{{\raise 0.5pt \hbox{\footnotesize $|$}}}}

\def\shah{\sqcup \hspace{-0.15em}\sqcup}   
\def\shah1{{\makebox[2.3ex][s]{$\sqcup$\hspace{-0.15em}\hfill $\sqcup$}}}   

\def\slb{{{\raise 0.5pt \hbox{\footnotesize $[$}}}}
\def\slcb{{{\raise 0.5pt \hbox{\footnotesize $\{$}}}}
\def\slp{{{\raise 0.5pt \hbox{\footnotesize $($}}}}   

\newcommand{\snorm}{{{\raise 0.5pt \hbox{\footnotesize $\|$}}}}   

\def\srb{{{\raise 0.5pt \hbox{\footnotesize $]$}}}}
\def\srcb{{{\raise 0.5pt \hbox{\footnotesize $\}$}}}}
\def\srp{{{\raise 0.5pt \hbox{\footnotesize $)$}}}}

\usepackage{graphicx}
\usepackage{xcolor}
\usepackage{amsfonts,amssymb,amsmath}
\usepackage{tikz}
\usetikzlibrary{arrows, decorations.markings}
\tikzset{>=latex}

\graphicspath{{figures/}}

\def\beq{\begin{equation}}
\def\eeq{\end{equation}}

\newcommand{\C}{\mathbb{C}}
\newcommand{\Z}{\mathbb{Z}}
\newcommand{\R}{\mathbb{R}}
\renewcommand\Re{\mathsf{Re}}
\renewcommand\Im{\mathsf{Im}}

\title{An optimally concentrated Gabor transform for localized time-frequency components\thanks{This work was supported by the European project UNLocX, grant n. 255931.
B. R. is with the Signal Processing Laboratory 2,  Ecole Polytechnique F\'ed\'erale de Lausanne (EPFL), Station 11, 1015 Lausanne, Switzerland. G. S. and H. L. are with Genesis, domaine du petit Arbois, BP 69, 13545 Aix en Provence, France. B. T. is with the Aix-Marseille Universit\'e, CNRS, Centrale Marseille, LATP, UMR7353, 13453 Marseille, France. C. W. is with the Faculty of Mathematics, University of Vienna, Austria. D. O. is with the Eftimie Murgu University, Romania.
}}

\author{Benjamin Ricaud         \and    Guillaume Stempfel \and   Bruno Torr\'esani \and Christoph Wiesmeyr\and H\'el\`ene Lachambre \and Darian Onchis 
}

\begin{document}
\maketitle

\begin{abstract}
Gabor analysis is one of the most common instances of time-frequency signal analysis. Choosing a suitable window for the Gabor transform of a signal is often a challenge for practical applications, in particular in audio signal processing. Many time-frequency (TF) patterns of different shapes may be present in a signal and they can not all be sparsely represented in the same spectrogram. We propose several algorithms, which provide optimal windows for a user-selected TF pattern with respect to different concentration criteria. We base our optimization algorithm on $l^p$-norms as measure of TF spreading. 
For a given number of sampling points in the TF plane we also propose optimal lattices to be used with the obtained windows.
We illustrate the potentiality of the method on selected numerical examples.
\end{abstract}
\section{Introduction}
The Gabor transform is widely used for the analysis of non-stationary signals, in particular in audio signal processing. 
In contrast to the classical Fourier Transform it allows for detection and estimation of localized time-frequency (TF) components or time-evolving frequency components. 
It is also used for modifying or filtering the signal in a limited TF region (through Gabor multipliers). 
The Gabor transformation coefficients are obtained by testing the signal against time frequency shifted copies of a window function.
The TF representation hence depends crucially on the choice of this predefined window and on the lattice giving the coordinates of the TF shifts. 

This freedom of choice however rises the important question of which window and lattice are the best for some given signal.
Indeed, this choice is crucial as the precision of the Gabor transform in time and frequency is limited by the uncertainty principle and it is not possible to have arbitrarily good time and frequency resolution at the same time.
It is in particular the decay of the chosen window on the time and frequency side that allows balancing time and frequency precision of the corresponding Gabor transform.
It often requires a delicate tuning from the user based on a visual appreciation of the spectrogram. 
Moreover, in some cases changing the size of the window is not sufficient for obtaining sharply concentrated information in the TF plane. 
Typical examples, where one needs to control more than the width of the window are signals that contain chirps, i.e. frequency components with time-varying frequency.

Although often neglected, the choice of the Gabor lattice is an important issue. 
It does not directly concern precision in time and frequency but more the computational effort needed to process the representation of the signal and its modification. 
A dense lattice with many sampling points implies a high redundancy of the representation and will therefore require a high computational cost and a heavy load for the computer memory. 
On the other hand, if one samples too coarsely, firstly the TF representation will be coarse, preventing the detection of fine details.
Secondly, if one chooses a bad sampling strategy, the Gabor transform may not even be invertible or it may result in an unstable inversion process.

The question of the adaptation of the window to a signal of interest had been considered in~\cite{Jaillet07time}, where it has been shown that quite often a global adaptation does not yield sensible results. 
A typical example is a musical signal which contains harmonic sounds from string or wind instruments and impact sounds from percussive instruments.
The associated TF patterns can be described as lines parallel to the time axis for constant frequencies and lines parallel to the frequency axis for percussive events.
It is hence impossible to find a window (and lattice) suitable for the analysis of the entire signal. 
However it is often useful to be able to represent one particular pattern with high precision inside the signal even if the cost is a worse representation in other TF regions.
Motivations are either for determining the characteristics of this TF shape, for highlighting it among other TF signatures often nearby located or for helping a systematic detection of every occurrence of this special pattern inside the signal.

In the following we propose an algorithm for obtaining a Gabor window and lattice adapted to a particular TF pattern via the solution of an optimization problem. 
According to our definition, a window is adapted when it concentrates the pattern optimally, in other words when the pattern becomes the sparsest possible in the TF plane. 
Our procedure hence minimizes the sparsity in the TF plane of a chosen component embedded in the whole signal. 
The idea is similar to the one of Jaillet-Torr\'esani~\cite{Jaillet07time}. 
However, our optimization scheme is non-parametric and optimizes over all possible functions (windows) in the signal space. 
In addition their sparsity measure (entropy) is a particular case of our more general $\ell^p$-norm measure. These quantities can be related via R\'enyi entropies~\cite{RicaudTorresani2}. 
The present work has been inspired by~\cite{Optimwin} in which an algorithm is proposed for building the maximally concentrated window fitting a given shape in the TF (Ambiguity) plane.
The optimization approach using $\ell^p$-norms is similar however there are many differences: in the motivation, in the algorithm and the more detailed theoretical study. 
In addition, our method is numerically faster as it relies on an iterative gradient method instead of the diagonalization of a large matrix (Gabor multiplier). 

To illustrate the idea, two spectrograms of the same signal using different Gabor windows are shown in Fig.~\ref{fig.aircraft}. 
On the left, the spectrogram has been computed with a standard window (Gaussian function). 
One can notice two frequency components evolving in time (thick red lines). 
The region, where we seek to have an optimal representation, is chosen to be one of the main components at a recording time of around $12$ seconds and has a length of around $1$ second. 
The spectrogram using the optimal window is shown on the right: the frequency components get thinner (sparser) between $10$ and $16$ seconds recording time.
Moreover, as the structure is resolved in more details, it reveals that the components are more complex than single lines.

\begin{figure}
\includegraphics[width=.49\linewidth]{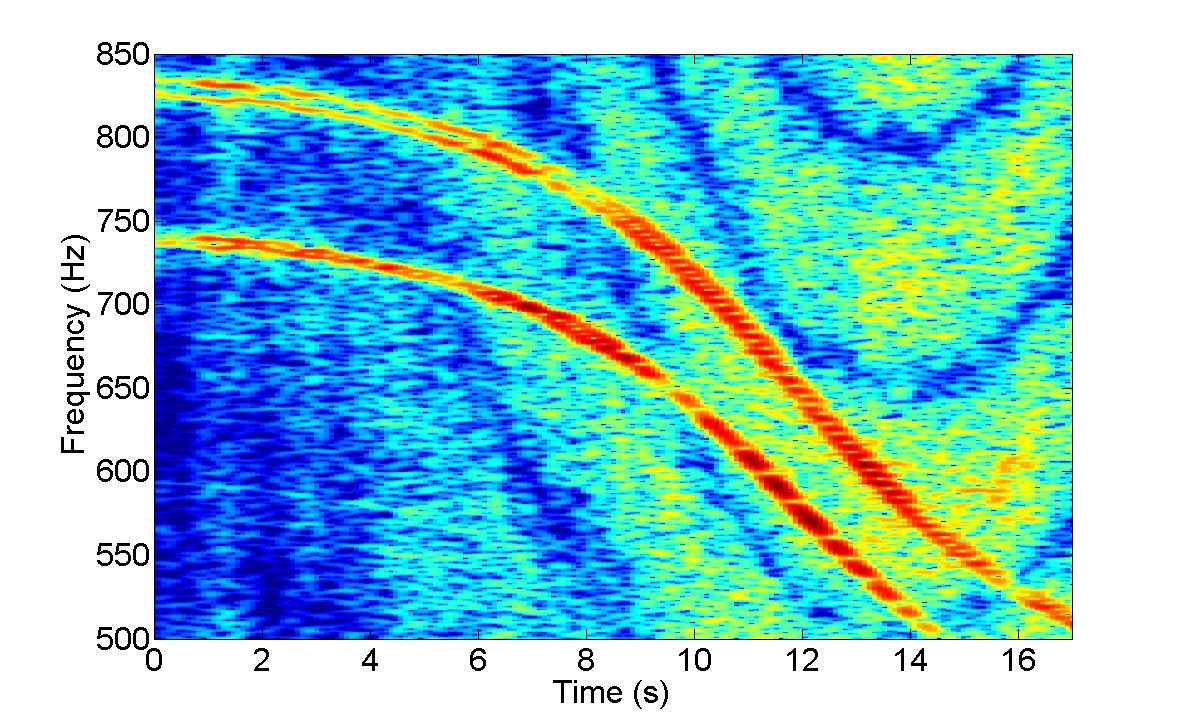}
\includegraphics[width=.49\linewidth]{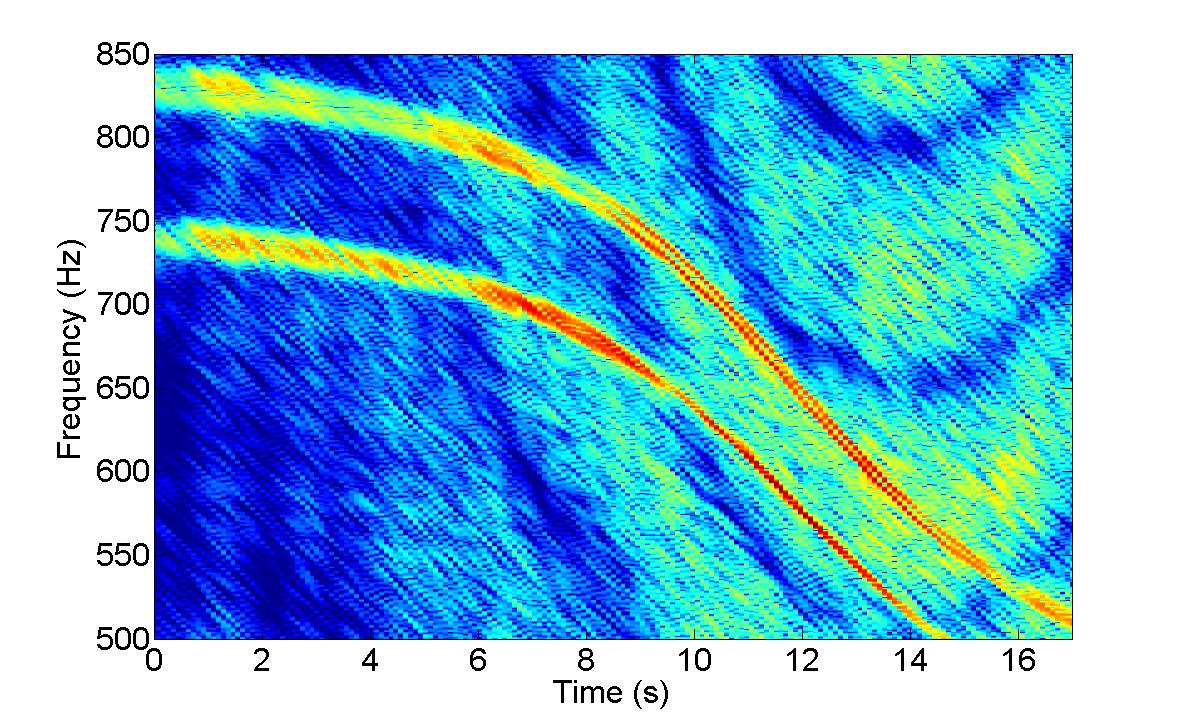}
\caption{Audio recording of an aircraft passing above the recording device (doppler effect). Left: spectrogram computed with a Gaussian window. Right: spectrogram computed with the optimal window. The optimization reveals the fine structure (i.e. 2 separated frequency components) in the TF domain where the optimization is performed, but kind of blurs the picture in other regions (namely the first 4 seconds) where another window would be better.}\label{fig.aircraft}
\end{figure}

In addition to the optimal window our method also provides an adapted lattice, suited for the optimal window. 
It is constructed to give an efficient placement of the Gabor atoms in the TF plane. 
For a fixed redundancy, the purpose is to provide the best tiling of the TF plane in order to represent as accurately as possible the energy distribution of the signal, with the given window.
A comprehensive study of the Gabor transform for general lattices can be found in~\cite{hosowi13}.
For example, a standard Gaussian window has a round-shaped ambiguity function and the optimal lattice would intuitively be hexagonal (like the problem of orange packing).
As in~\cite{hosowi13}, we have chosen three tuning parameters to adapt the lattice: the time and frequency step sizes (as in standard lattices) and the shearing.
Later we will give explicit sampling strategies for a class of windows called \emph{generalized Gaussians}, as this class in particularly adapted to the representation of audio signals. 

\section{Preliminaries}\label{prelim}

Here we will outline the theory of discrete Gabor transforms, for a more complete picture of the theory and algorithms we refer to \cite{fest98}.
An in-depth investigation of the computational complexities of the finite discrete Gabor transform can be found in \cite{so12,hosowi13}.
The continuous theory is described in \cite{gr01,ma09-1}.
%

To be able to define the Gabor transform we make use of the \emph{translation} and \emph{modulation} operators denoted by $T$ and $M$ respectively.
For $x, \xi \in \left\{ 0,\dots,N-1 \right\}$  we define them as
\begin{align}
  T_x g(t) &= g(t-x) \text{ and} \\
  M_\xi g(t) &= e^{2 \pi i \xi t/N}g(t),
  \label{}
\end{align}
for any $g\in \C^N$.
At this point we note that all the indices are to be interpreted modulo $N$ throughout the rest of this contribution (periodic boundary conditions).

With this notation we define the Gabor family with respect to some sampling set $\Lambda \subseteq \bbZ_N^2$ and some window function $g$ as
\begin{equation}
  \mathcal G(g,\Lambda) = \left\{ M_\xi T_x g:\, (x,\xi)\in \Lambda \right\}.
  \label{}
\end{equation}
The set $\bbZ_N^2$ is called \emph{time-frequency plane}, as every point corresponds to a \emph{time-frequency shift} of the window $g$.
The Gabor transform ${\cal V}_gf$ of a given signal $f$ with respect to $g$ and sampling set $\Lambda$ is given by its transform coefficients:
\begin{equation}
  \left\{{\cal V}_gf(x,\xi)=\inner{f}{M_\xi T_x g}: \, (x,\xi)\in \Lambda  \right\},
\end{equation}
where $\langle\cdot,\cdot\rangle$ stands for the scalar product in $\C^N$.

It is possible to perfectly reconstruct the signal from its coefficients if and only if ${\cal G}(g,\Lambda)$ is a frame \cite{ch08}, i.e. when there exist two bounds $A,B>0$ such that for all $f$ in the signal space
\begin{equation*}
  A\|f\|_{\ell^2(\C^N)}^2\le \|{\cal V}_gf\|_{\ell^2(\Lambda)}^2\le B \|f\|_{\ell^2(\C^N)}^2 .
\end{equation*}
The frame property of the family ${\cal G}(g,\Lambda)$ is equivalent to the existence of a (possibly non-unique) dual window $g_d$, such that the inverse transform is given by
\begin{equation}
  f = \sum_{(x,\xi)\in \Lambda}\inner{f}{M_\xi T_x g} M_\xi T_x g_d.  \label{eq:frameexp}
\end{equation}
We call a frame tight if the window is dual to itself.

In the definition of frames we have been overly careful with denoting norms, throughout the rest of the paper we will denote the $p$-norm by $\|\cdot\|_p$, as the underlying space will always be clear.

\noindent{\bf The space of time-frequency centered windows $\cal C$.}  In the following, we will use $p$-norms as measure of optimality. Due to the property that for any $\xi,x\in \Lambda$ and $g\in \ell^2$, $\|M_\xi T_x g\|_p=\|g\|_p$, any time-frequency shifted copy of a solution will be a solution. To avoid this ambiguity, we introduce the set of time-frequency centered windows $\cal C$ in which the optimization will be processed. To define this set we first define what is a TF centered window.
Let us introduce the mean value in time of $g\in\C^N$:
\begin{equation}\label{meant}
\mu_t(g)=\arg\left(\sum_{n=0}^{N-1} e^{2i\pi \frac{n}{N}}|g(n)|^2\right),
\end{equation}
which is the definition of the mean value for functions on periodic domains (von Mises' mean).
The mean value in frequency $\mu_f(g)$ is obtained by replacing $g$ by its Fourier transform in~\eqref{meant}. The TF centered version of $g$ reads
$$M_{-\mu_f}T_{-\mu_t}g.
$$
The set $\cal C$ of TF centered functions is 
$${\cal C}=\left\{g\in\ell^2(\C^N), \mu_t(g)=\mu_f(g)=0\right\}.
$$
\section{Optimal window}


For a given Gabor lattice, we call a window optimal if it maximally concentrates the energy of the signal in the TF plane. 
Given a window with unit energy we use the $\ell^p$-norm with $p>2$ (resp. $p<2$) as a measure of spread.
Maximizing ($p>2$) or minimizing ($p<2$) this norm leads to sparse representations (see Ricaud-Torr\'esani~\cite{RicaudTorresani2}).
In the following we will choose $p>2$.
In this case the quantity $ \| {\cal V}_gf\|_p^p$ is differentiable with respect to $g$ and we can use a standard gradient method to find the optimum (here the maximum).

The first optimization problem reads as
\begin{equation}\label{optprob}
  g_{\rm opt}=\mathop{\rm Argmax}_{g\in{\cal C},\|g\|_2=1} \| {\cal V}_gf\|_p^p=\mathop{\rm Argmax}_{g\in{\cal C},\|g\|_2=1}  \sum_{\xi,x\in\Lambda}|\langle M_\xi T_xg , f\rangle |^p,
\end{equation}
with $p> 2$. Under this latter condition, the $p$-norm is a measure of sparsity and maximizing this quantity leads to a sparser TF representation of $f$. The solution $g_{\rm opt}$ of~\eqref{optprob} (whenever it exists) is the optimal window. However, the set of functions $g$ such $\|g\|_2=1$ is not concave.
Hence the above formulation does not lead to a concave problem and convergence issues have to be investigated. In addition, any time frequency shifted copy $M_\xi T_x g_{\rm opt}$ of the optimal window is also an optimal window. This is why we restrict the space of windows to the ones in $\cal C$, centered in time and frequency. We also apply a TF centering to the function $f$, see Sec.~\ref{preprocessing} for the signal preprocessing.

It is also interesting to look at the first order term when $p$ is close to 2 and $\|f\|_2=1$: it yields the entropic optimization problem (c.f. Remark~\ref{rementropy}),
\begin{equation}\label{optprobentropic}
\mathop{\rm Argmax}_{g\in{\cal C},\|g\|_2=1}  \sum_{\xi,x\in\Lambda}-|\langle M_\xi T_xg , f\rangle |^2 \ln |\langle M_\xi T_xg , f\rangle |^2,
\end{equation}
which also maximizes the sparsity and has been used in Jaillet-Torresani~\cite{Jaillet07time} for the optimization of a parametric window.
We may also generalize their approach in the following way. We introduce the parametrized window (chirped and dilated Gaussian-like window),
\begin{equation}\label{ParamGaussian}
g_{\sigma,s}(t)=\sqrt[4]{\frac{2}{N\sigma}}e^{-\pi\frac{t^2}{N\sigma}+i \pi s t^2\frac{N+1}{N}},
\end{equation}
where $\sigma>0$ is the width in time and $s\in\R$ is the chirping parameter. The number of samples of the signal is $N$. The optimization problem becomes:
\begin{equation}\label{optprobparam}
\mathop{\rm Argmax}_{\sigma,s} \| {\cal V}_{g_{\sigma,s}}f\|_p^p.
\end{equation}
The projection on the non-concave set disappears but still the cost function may not be concave in general.

Eventually, another optimization problem which is concave is the following:
\begin{equation}\label{optprob2} 
\mathop{\rm Argmax}_{g} \| {\cal V}_gf\|_p^p-\lambda \|h-g\|_2^2, \qquad p>2, 
\end{equation}
where $h$ is for example a standard Gaussian function centered in time and frequency. 
An elliptic Gaussian or any other well-concentrated function may also do the trick. 
The parameter $\lambda>0$ allows one to tune the shape of the optimal window. 
This formulation as well as~\eqref{optprobparam} are of interest for practical applications: one often wants to obtain a window which is not too eccentric. Indeed, in most applications, TF localization is as important as sparsity. 
Localization in the TF plane is the reason why one chooses the Gabor transform to analyze a signal. 
If this property is lost the Gabor transform coefficients are of little interest.
For example an optimal window whose TF transform is very elongated in one direction of the plane won't be a good optimal window in practice as the localization in that direction will be poor (although the localization in the perpendicular direction is extremely good). The interpretation of the spectrogram will be difficult and filtering based on localized TF regions will not be efficient. 
In~\eqref{optprob2} the localization is constrained by the choice of $\lambda$ and in~\eqref{optprobparam} it is constrained by limits on the values $\sigma$ can take (upper and lower bound).

Before going to the details of the window optimization, let us point out some important remarks.
\begin{remark}
The optimization is done on a fixed Gabor lattice which may not be the ideal one for the detection of the TF pattern. 
Hence, after optimizing the window the next step of our optimization process is to adapt a new lattice $\Lambda^{(1)}$ to the optimal window $g_{\rm opt}$. 
This new sampling strategy depends on the shape of the ambiguity function of the window.
It seeks to tile the TF plane optimally with such atoms, a detailed explanation can be found in Section \ref{optimallattice}.
\end{remark}
\begin{remark}
We enforce sparsity of the Gabor representation from an analysis point of view.
Therefore it is not necessary to compute a dual Gabor window or an inverse Gabor transform during the optimization process.
\end{remark}

\subsection{The non-concave, non-parametric case}
Problem~\eqref{optprob} is non concave but satisfies the hypotheses stated in~\cite[Sec. 5]{Attouch11}, whenever $p$ is a rational number. To prove that, let us first introduce $\widetilde{g}=(\Re(g),\Im(g))\in\R^{2N}$ and rewrite~\eqref{optprob} as:
\begin{align}\label{newoptprob}
\mathop{\rm Argmin}_{\widetilde{g}\in \R^{2N}} {\cal F}(\widetilde{g})+i_C(\widetilde{g}),
\end{align}
where 
$${\cal F}(\widetilde{g})=-\sum_{\xi,x\in\Lambda}\left(\sqrt{\Re \left(\langle g ,T_x M_\xi f\rangle\right)^2+\Im \left(\langle  g ,T_x M_\xi f\rangle\right)^2} \right)^p,\qquad p>2,
$$
and $i_C$ is the indicator function of the closed set $C=\{g: \|g\|_2=1\}$. 
Introducing 
$$(K_{\xi,x}(\widetilde{g},f))^2=\Re \left(\langle g ,T_x M_\xi f\rangle\right)^2+\Im \left(\langle  g ,T_x M_\xi f\rangle\right)^2,
$$
and replacing $g$ and $T_x M_\xi f$ by $\Re(g)+i \Im(g)$ and $\Re(T_x M_\xi f)+i \Im(T_x M_\xi f)$ in the above expression shows that $K_{\xi,x}(\cdot,f)$ is a semialgebraic function. The right hand side is a polynomial in the elements of vector $\widetilde{g}$.

Then ${\cal F}=-\sum_{\xi,x\in\Lambda}\left(K_{\xi,x}(\cdot,f)\right)^p$ and for any integer $p$, the real function $\cal F$ is a polynomial in the elements $\{K_{\xi,x}\}$. Hence $\cal F$ is semialgebraic which implies that it is a lower semicontinuous Kurdyka-Lojasiewicz function (see~\cite{Attouch11} for more details). It can be generalized to any $p$ rational by writing $p=c/d$ and replacing $K_{\xi,x}$ by $K_{\xi,x}^{d}$. Concerning the lower bound on the cost function, ${\cal F}$ is strictly negative and  ${\cal F}=0$ (maximum) if and only if $g=0$ since $\{T_xM_\xi f\}$ is a frame. The constraint $\|g\|_2=1$ ( and $\{T_xM_\xi f\}$ is a frame) implies that the function ${\cal F}+i_C$ is bounded from below. Moreover ${\cal F}$ is differentiable and has a $L$-Lipschitz continuous gradient. To see that one has to differentiate with respect to $\Re(g)$ and $\Im(g)$. It is more convenient to differentiate (in the complex sense) with respect to $\overline{g}\in\C^N$ and use Cauchy-Riemann equations, this is what we do in the following. We have (see below for details of the calculation):
\begin{align}
\nabla_{\overline{g}} \| {\cal V}_{g}f\|_p^p=\frac{p}{2}G_{f,\Lambda,\Gamma}(g),
\end{align}
where $G_{f,\Lambda,\Gamma}$ denotes the following time-frequency masking operator:
\begin{align}\label{eq:tfmasking}
G_{f,\Lambda,\Gamma}(g)=\sum_{\xi,x\in\Lambda} \Gamma_g(x,\xi) \langle g,  M_\xi T_x f\rangle  M_\xi T_x f.
\end{align}
with mask $\Gamma_g(x,\xi)=  |\langle M_\xi T_x f , g\rangle |^{\frac{p}{2}-1} $.
In~\cite{Attouch11} the authors show that under these conditions the forward-backward algorithm converges to a critical point of ${\cal F}+i_C$. 
This procedure iteratively computes
\begin{equation}\label{gradascent}
g_{i+1/2}=g_i+\gamma \nabla_{\overline{g}} \| {\cal V}_{g_i}(h)\|_p^p,
\end{equation}
where $\gamma\in (0,L)$ and then projects the current iterate to the $\ell^2$-sphere by
$$g_{i+1}=\frac{g_{i+1/2}}{\|g_{i+1/2}\|_2}.
$$

\begin{remark}\label{reminfinity}
Let us first treat the particular case of $p=\infty$. We have the following result:
$$\lim_{p\to\infty}\left(\sum_{\xi,x}|\langle M_\xi T_xg , f\rangle |^p\right)^{1/p}=\max_{\xi,x}|\langle M_\xi T_x g , f\rangle |.
$$
The maximum of the scalar product is attained when  $M_\xi T_xg=f$. The optimal window is hence equal to the signal (up to a translation and modulation).
\end{remark}

The convergence results are summarized in the following theorem.
\begin{theorem}\label{Th1}
The forward-backward algorithm associated to the optimization problem~\eqref{optprob} converges to a critical point $g_{\rm opt}$ which depends on $p$ in the following manner:
for any finite $p> 2$, $g_{\rm opt}$ is the limit of $g_i$ when $i$ tends to infinity of the gradient ascent~\eqref{gradascent} with
$\nabla_{\overline{g}} \| {\cal V}_{g}(h)\|_p^p=\frac{p}{2}G_{h,\Lambda,\Gamma}(g)$ where $\Gamma_g(x,\xi)=  |\langle M_\xi T_xg , f\rangle |^{\frac{p}{2}-1} $.
\end{theorem} 

\begin{proof}
Convergence issues have been treated in the beginning of this section. we concentrate on the derivation of the gradient. 
Introduce $A(g,\overline{g})=|\langle M_\xi T_x g , f\rangle |^2$, its derivative with respect to $\overline{g}$ (taking $g$ as independent of $\overline{g}$, see~\cite{ComplexGradient} and refs. therein) is:
$$\nabla_{\overline{g}}A(g,\overline{g})=\langle M_\xi T_x g , f\rangle T_x M_\xi f.
$$
One can write ($p>2$)
$$\nabla_{\overline{g}}\sum_{\xi,x}|\langle M_\xi T_x g , f\rangle |^p=\sum_{\xi ,x }\nabla_{\overline{g}}A^{p/2}(g,\overline{g})=\frac{p}{2}\sum_{\xi ,x } A^{\frac{p}{2}-1} (g,\overline{g})\nabla_{\overline{g}}A(g,\overline{g}).
$$
\end{proof}

\begin{remark}[Near $p=2$, the Shannon entropy is not far]\label{rementropy}

Introduce $\eta=2-p$. One can write
$A^\eta=\exp(\eta\ln A)=1+\eta\ln A+\eta^2\ln^2 A+\cdots$.
When $p$ is close to 2 the $\ell^p$-norm is:
\begin{align*}
\sum_{\xi,x}|\langle M_\xi T_xg , f\rangle |^p&=\sum_{\xi,x}|\langle M_\xi T_xg , f\rangle |^2 |\langle M_\xi T_xg , f\rangle |^\eta\\
&=1+\frac{\eta}{2} \sum_{\xi,x}|\langle M_\xi T_xg , f\rangle |^2 \ln|\langle M_\xi T_xg , f\rangle |^2+{\cal O}(\eta^2)
\end{align*}
For small $\eta$ the second term is the leading one in the optimization process. 
The expression of its gradient is $\nabla_{\overline{g}}(A\ln A)=(1+\ln A) \nabla_{\overline{g}}A $.
Note that if $\{T_xM_\xi f\}$ is a tight frame, $\sum_{\xi,x} \nabla_{\overline{g}}A(g,\overline{g})= C g$, ($C>0$ being a constant depending on the frame). 
The leading term in the gradient will be
$$\frac{\eta}{2}\left( Cg+ G_{h,\Lambda,\Gamma}(g)\right),
$$
where the operator $G_{h,\Lambda,\Gamma}(g)$ has the following mask:
$$\Gamma_g(x,\xi)=\ln|\langle M_\xi T_xg , f\rangle |^2.
$$
\end{remark}

\subsection{The non-concave parametric case}

In practice the Gabor transform is often used to visually estimate the behavior of some TF component.
The analysis window $g$ must not be too exotic in order to keep the intuitive interpretation of the transform coefficients.
A first requirement is that the window should be intrinsically concentrated in time and frequency, a second one is that it should be localized in only one connected region of time and frequency. 
A third requirement is that the shape of the window should be close to an ellipsis in the TF plane. It would allow the window to have a privileged direction in the TF plane but not only restricted to either time or frequency.
A sensible way to obtain such a function is to do the optimization in the subset of chirped Gaussians.

Let us introduce the subclass $\cal G$ of chirped Gaussians in $\C^N$ (or generalized Gaussians).
A function in this class is of the form $M_\xi T_x \phi_{\sigma,s}$, where $\phi_{\sigma,s}=g_{\sigma,s}$ of~\eqref{ParamGaussian}, for all possible values of the dilation parameter $\sigma > 0$ and chirping parameter $s \in \R$.
Note that $\|M_\xi T_x\phi_{\sigma,s}\|_2=1$.
We want to maximize the following quantity over $\sigma$ and $s$:
$$
\|{\cal V}_{\phi_{\sigma,s}}f\|_p^p, \qquad p>2.
$$
As in the previous section the problem is not concave, it depends on $f$ in a non trivial way even after factoring out the constant phase indeterminacy. Actually it is interesting to notice that if $f$ is the sum of two Gaussians with same mean but with two different variances, the function $|\langle\phi_{\cdot,s=0},f\rangle|:\sigma\to\R_+$ has two maxima (hence non-concave). 

If we assume that $f$ is such that $|\langle\phi_{\cdot,\cdot},f\rangle|$ is bi-concave (concave in each variable separately) standard optimization algorithms will converge but may converge to a local maximum~\cite{biconvex07}. 
In the following we will add this property as a condition on our function $f$, which in practice covers many TF patterns. One has to keep in mind that there are also many examples of $f$ which do not lead to concavity or bi-concavity. As a few examples we can list:
\begin{itemize}
\item the above example of a sum of 2 Gaussians with different spreading, 
\item a sum of 2 localized functions with different chirping values, forming a "cross" pattern in the TF plane,
\item broad band signals.
\end{itemize}
The convergence of the optimization process is not guaranteed for these patterns.

Using the results of the previous section, the gradient is given by
\begin{align}
\frac{\partial \|{\cal V}_{\phi_{\sigma,s}}f\|^p}{\partial \sigma}=p\ \Re[\langle G_{f,M}\phi_{\sigma,s},\frac{\partial \phi_{\sigma,s}}{\partial \sigma}\rangle],\\
\frac{\partial \|{\cal V}_{\phi_{\sigma,s}}f\|^p}{\partial s}=p\ \Re[\langle G_{f,M}\phi_{\sigma,s},\frac{\partial \phi_{\sigma,s}}{\partial s}\rangle],
\end{align}
where $G_{f,\Lambda,\Gamma}$ is the operator defined in \eqref{eq:tfmasking} with window $f$ and mask $\Gamma=|\langle M_\xi T_xf , \phi_{\sigma,s}\rangle |^{p-2}$. One can see it by writing for $p>2$ even, any window $g$ and function $f$:
\begin{align*}
\sum_{\xi,x}|\langle M_\xi T_xg , f\rangle |^p&=\sum_{\xi,x}\langle M_\xi T_xg , f\rangle^{p/2} \langle f,M_\xi T_xg \rangle^{p/2}
\end{align*}
and 
$$\frac{\partial}{\partial s}\langle M_\xi T_x\phi_{\sigma,s} , f\rangle^{p/2}=\frac{p}{2}\langle \frac{\partial}{\partial s}\phi_{\sigma,s} , T_{-x}M_{-\xi} f\rangle \langle \phi_{\sigma,s} , T_{-x}M_{-\xi} f\rangle^{p/2-1}.
$$
At each step of the gradient method, one computes:
\begin{align}
\sigma_{n+1}&=\sigma_n+\alpha p\ \Re[\langle G_{f,\Lambda,\Gamma^{(n)}}\phi_{\sigma,s}^{(n)},\frac{\partial \phi_{\sigma,s}^{(n)}}{\partial \sigma}\rangle]\\
s_{n+1}&=s_n+\alpha p\ \Re[\langle G_{f,\Lambda,\Gamma^{(n)}}\phi_{\sigma,s}^{(n)},\frac{\partial \phi_{\sigma,s}^{(n)}}{\partial s}\rangle]
\end{align}
and $\phi_{\sigma,s}^{(n+1)}=\phi_{\sigma_{n+1},s_{n+1}}$.

The derivatives of the function $\phi_{\sigma,s}$ with respect to $\sigma$ and $s$ read:
$$\frac{\partial \phi_{\sigma,s}}{\partial \sigma}(t)= \left(\frac{\pi t^2}{\sigma^2 N}-\frac{1}{4\sigma}\right)\phi_{\sigma,s}(t),\qquad
\frac{\partial \phi_{\sigma,s}}{\partial s}(t)=  \frac{i\pi}{N} t^2 \phi_{\sigma,s}(t).
$$
\begin{remark}[Control over the parameters]
At each step we can also control the value of the parameters $\sigma$ and $s$. Additional constraints may be added in the loop. For example if one wants an optimal window with a limited spread in one direction, the value of $\sigma$ may be forced to stay within some predefined bounds.
\end{remark}
\begin{remark}[Gradient step and the BFGS optimization scheme]
Experiments show that it is difficult to find an appropriate gradient step to ensure the convergence of the algorithm (large steps lead to divergence, small steps to extremely slow convergence). Even an adaptive step is not satisfactory because of the shape of the function.
Thus, we decided to resort to a second-order-like method. Since the computation of the second derivatives is hard, we used a BFGS (Broyden-Fletcher-Goldfarb-Shanno, see e.g. \cite{goldfarb70}) technique.
Basically, BFGS, like other Quasi-Newton procedures, is an iterative algorithm which is an approximation of Newton's method, without the need to compute the inverse of the hessian matrix.
This hessian matrix is estimated using the first-order gradient computed for each iteration.
The larger the number of iterations is, the more accurate the hessian estimation will be.
A line search is performed in order to find an acceptable step size (actually an $\alpha$ value which satisfies Wolfe conditions).
The algorithm iterates until the gradient norm becomes small enough, which indicates that a local optimum has been found.
Experimentally, under this form, the algorithm converges rather quickly. The BFGS procedure is designed for tackling concave or convex problems. Even though our objective function was not concave, it did not appear to be a problem in the computations. In our simulations, if the starting point of the optimization was reasonably chosen (typically a Gaussian time-support around half a second, with a chirping parameter equal to zero), the algorithm always converged to an acceptable solution. However, because of the non-concavity of the function, we cannot guarantee the quality of obtained solutions theoretically.
\end{remark}
\subsection{The concave case}
The optimization problem~\eqref{optprob2} is concave and the solution may be obtained using a gradient ascent. Starting from any initial $g_0$ the iterative algorithm computes at step $i+1$:
\begin{equation}
g_{i+1}=(1-\gamma\lambda)g_i+\gamma\lambda h+\gamma \nabla_{\overline{g_i}}  \| {\cal V}_{g_i}f\|_p^p,
\end{equation}
where $\gamma>0$ is the gradient step.
The expression of the gradient $\nabla_{\overline{g_i}} \| {\cal V}_{g_i}f\|_p^p$ depends on $p$ and may be found in Th.~\ref{Th1}.

Preliminary tests were not convincing in term of signal representation compared to the other cases. Moreover choosing $\gamma$ and $\lambda$ is delicate. That is why we concentrated more on the non-convave parametric problem which was more promising.
We present it only in order to show that it is possible to find a concave optimization problem for the window optimization. 

\subsection{Localizing a pattern of interest in the TF plane (data preprocessing)}\label{preprocessing}
The preprocessing of the data  before optimization consists of the 3 following steps.

\noindent{\bf Localization.} Let us suppose that there is a localized TF pattern of interest embedded in the signal $f$.
A Gabor transform with standard lattice and standard Gaussian window can be used to get a first crude estimate of its TF location and spreading.
In order to optimize the Gabor window with respect to this particular pattern we may proceed as follow.

In the first step we extract this pattern by multiplying all the TF coefficient not in the area of interest with zeros and reconstruct (compute the inverse Gabor transform).
The resulting signal, lets call it $h$, contains only the pattern (or more precisely a good approximation of it) to be represented in sparser way. 

\noindent{\bf TF centering.} By translating and modulating $h$ we can center the selected component at the origin of the TF plane. Since the windows belong to the set of TF centered functions, doing so avoids unnecessary TF shifts in the optimization process. It also allows data reduction (see next step). The centering is performed according to the formula given in Sec.~\ref{prelim}.

\noindent{\bf Data reduction.} To further simplify the problem one can reduce the number of samples of $h$, which is at this point still equal to the number of samples of the original function.
Often, it is possible to neglect part of the samples in time if the region of interest has a small time width.
If the pattern covers only a small portion of the frequencies it is possible to downsample $h$ without deleting information.


\section{Adapted lattice}\label{optimallattice}

The choice of a Gabor lattice may be as important as the choice of a window. In this section we describe firstly what is the set of all possible Gabor lattices. Secondly we present what is a good lattice for a given window, in the sense that it optimally packs the windows in the TF plane (for a fixed transform redundancy). Thirdly, we show how to compute the Gabor transform with a non-standard lattice as fast as the one having a standard lattice.

In the discrete setting with signal length $N$ any Gabor lattice $\Lambda$ can be expressed by
\begin{equation}\label{eq:lattnf}
  \begin{pmatrix}
    a & 0 \\
    s & b
  \end{pmatrix}
  \cdot \bbZ_N^2,
\end{equation}
for some uniquely determined $a,b$ which must divide $N$ and $s=k \; b/\gcd(L/a,b)$ for some integer number $k$~\cite{hahotowi12}.
The time frequency plane is given by $\bbZ_N^2$, which is at the same time the finest (and densest) possible lattice.
This means that we can generate all possible lattices by choosing $a$, $b$ and $s$ satisfying the conditions above.
The redundancy of a lattice of the form \eqref{eq:lattnf} is given by $R = L/(ab)$, regardless of the parameter $s$.

As we have seen previously, dilated and chirped Gaussians play a central role in our optimization scheme.
They possess good localization properties, are easy to control and allow for a good interpretation of the spectrogram.
In order to choose the optimal lattice it seems sensible to seek to minimize the condition number of the resulting Gabor system, i.e. the ratio between the upper and lower frame bound $B/A$.
Geometrical arguments have led to the conclusion that hexagonal patterns are the most appropriate for standard Gaussian functions in the continuous case \cite{best03}.

For windows which are dilated and chirped versions of the standard Gaussian, the adapted lattice (which has the same properties as the hexagonal one for the round Gaussians) may be computed knowing the dilation and chirping parameters.
We refer to~\cite{gr01,deon12} for the continuous case and to~\cite{hosowi13,hahotowi12} for the discrete setting.
It is shown in~\cite[Ex. 9.4.1]{gr01} that two Gabor transforms with two different lattices are equivalent provided their respective windows are related by some dilation and chirping; the values of which are directly proportional to the lattice step and shearing values.

Under the assumption that hexagonal sampling is optimal for the standard Gaussian, a careful analysis of the connection between continuous and discrete Gabor transforms yields the theoretical optimal sampling points for a window of the form $\phi_{\sigma,s}$ defined in~\eqref{ParamGaussian}:
\begin{equation}
  \Lambda = \sqrt{\frac{N}{R}}
  \begin{pmatrix}
    1&0\\
    s&1
  \end{pmatrix}
  \begin{pmatrix}
    \sqrt{\sigma} & 0 \\
    0 & 1/\sqrt{\sigma}
  \end{pmatrix}
  \begin{pmatrix}
    \frac{\sqrt[4]{3}}{\sqrt{2}} & 0 \\
    \frac{1}{\sqrt[4]{3}\sqrt{2}}& \frac{\sqrt{2}}{\sqrt[4]{3}}
  \end{pmatrix} \cdot \bbZ_N^2.
  \label{latticetransformation}
\end{equation}
The 3 successive matrices correspond, from left to right, to the shearing, the dilation and the hexagonal shaping of the $N\times N$ square lattice $\Z_N^2$. 
Several entries of these matrices are irrational numbers which will of course not satisfy the requirements we mentioned above. Therefore we make a first approximation by choosing feasible $a$, $b$ and $s$ of $\Lambda$ which are closest to the true values.

In the following we show in more details how the shearing operation is related to chirping in the signal domain.
Moreover, we demonstrate how \emph{a good approximation} of the Gabor transformation with a sheared lattice may be computed using the regular Gabor transform. 
Since we intend to compute optimal windows using non-regular lattices in Sec.~\ref{Loopoptim}, it is important to have a fast computation of the non-separable Gabor transform.

In the discrete setting, due to periodic conditions, the chirp functions are restricted to the ones which revolve around the TF plane $s$ times for $s\in \bbN_0$.
Let $U_s$ denote the chirping operator on a signal $g\in\C^N$:
\begin{equation}\label{chirping}
 U_sg(n)=e^{i\pi s n^2\frac{N+1}{N}}g(n),
 \end{equation}
and let $\Lambda,\Lambda'$ be two Gabor lattices related by a shearing of $s$:
 \begin{equation}\label{defGabLat}
 x'=x,\qquad \xi'=\xi+xs \qquad \forall (x',\xi')\in\Lambda', \ (x,\xi)\in\Lambda.
 \end{equation}
For an integer value of $s$ one has~\cite{hosowi13}
\begin{equation}
 \langle f, M_{\xi'}T_{x'} g\rangle=e^{is\pi x^2\frac{N+1}{N}}\langle U_{-s}f, M_{\xi}T_{x} U_{-s}g\rangle.
\end{equation}
To prove the equation above one needs to interchange chirping and translation.
This is only possible for integer chirping parameters (periodic boundary conditions).

%
%
Under the same restriction, the operator $G_{g,\Lambda',\Gamma}$ involved in the optimization procedure\footnote{This operator is a Gabor multiplier if the frame is tight.}, which is associated to a sheared lattice may be computed using:
 \begin{align*}
 G_{g,\Lambda',\Gamma}(f)&=\sum_{x',\xi'}\Gamma(x',\xi')\langle f, M_{\xi'}T_{x'} g\rangle  M_{\xi'}T_{x'} g\\
 &=\sum_{x',\xi'}\Gamma(x',\xi')\langle  U_{-s}^*U_{-s}f, M_{\xi'}T_{x'}  U_{-s}^*U_{-s}g\rangle  U_{-s}^*U_{-s}M_{\xi'}T_{x'} g\\
 &=U_{-s}^*\sum_{x,\xi}\Gamma(x,\xi+sx)\langle U_{-s}f, M_{\xi}T_{x} U_{-s}g\rangle M_{\xi}T_{x} U_{-s}g,
 \end{align*}
 where
$$\Gamma(x',\xi')=|\langle f, M_{\xi'}T_{x'} g \rangle |^{\frac{p}{2}-1}=|\langle U_{-s}f, M_\xi T_x U_{-s}g \rangle |^{\frac{p}{2}-1}. $$
The phase factor disappears in the expression of the operator. As a consequence,  $G_{g,\Lambda',\Gamma}(f)$ may be calculated with a standard lattice $\Lambda$ but with a chirped version of $f$ and $g$.

As mentioned above, for the finite discrete setting $s\in \bbN$ is necessary to make sure that the chirp does an integer number of revolutions around the TF plane.
The restriction comes from the periodic structure of the discrete signal domain.
However, for concentrated and TF centered signal and window function the boundary terms of the TF representation become negligible, as we will show in the following. 
Hence we can act as if we were without any periodicity requirements and choose whatever $s\in \R$ is needed. 

Let $f$ and $g$ be well localized signals around the origin ($t=N/2$). In the following we will show that the energy of the TF representation ${\cal V}_g f$ is negligible around the boundaries. Assume that for all $t\notin[N/2- N/8,N/2+N/8]$, $|f(t)|\le \varepsilon$ and $|g(t)|\le \varepsilon$, where $\varepsilon$ is small. Translating the window by $x$ gives $T_xg$ centered around $t=N/2+x$ and as a consequence the overlap between $f$ and $T_xg$ is decreased. 
Then for $x\notin [- N/4,N/4]$ and for all $\xi$,
\begin{align}
|\langle f, M_{\xi}T_{x} g\rangle|&=|\sum_{t=0}^{3N/8}\varepsilon M_{\xi}T_{x} g(t)+\sum_{t=3N/8}^{5N/8} f(t)\varepsilon+\sum_{t=5N/8}^{N-1}\varepsilon M_{\xi}T_{x} g(t)|\nonumber\\
&\le \varepsilon \left(\|f\|_1+ \|M_{\xi}T_{x} g\|_1\right)=\varepsilon \left(\|f\|_1+ \|g\|_1\right)\\
&\le \varepsilon \sqrt{N} \left(\|f\|_2+ \|g\|_2\right).\nonumber
\end{align}
The same results hold along the frequency axis, it can be shown by replacing $f$ and $g$ by their Fourier transform.
As a consequence, away from the center of the TF plane the energy of the signal representation may be neglected as well as its influence on the optimal window calculation.

Concerning dilations $\sigma$, the same property of concentration around the origin allows us to dilate as in the continuous case, without taking care of what could happen at the boundaries. 
The value of $\sigma$ is directly given as the output of the window optimization and the lattice parameters $(a,b)$ are replaced by $(\sqrt{\sigma}a,b/\sqrt{\sigma} )$ plus an extra term coming from the redundancy (see Eq.~\eqref{latticetransformation}).

\begin{remark}
If the window is given by the unconstrained optimization problems~\eqref{optprob} or~\eqref{optprob2}, the lattice parameters may be fixed by the following process. Using formula~\eqref{ParamGaussian}, a set of dilated and chirped (periodized or truncated) Gaussian is constructed. The scalar product between the optimal window and the elements of the set is computed. The maximal value obtained for the scalar product gives the dilated and chirped Gaussian closest to the optimal window. The parameters of this function (dilation $\sigma$, chirping $s$) are then used to construct the adapted lattice.
Of course, the shape of the optimal window may be far from an ellipsis in the TF plane, especially when using~\eqref{optprob} and one has to be aware of this limitation.
\end{remark}

\section{Alternating optimal window and adapted lattice}\label{Loopoptim}

We point out that the computation of the optimal window is done with a fixed given lattice. 
The output of the previous step (section~\ref{optimallattice}) gives a lattice adapted to the optimal window. 
But this window is optimal (in terms of concentration of the TF representation of $f$) for a different lattice. It is interesting to see now what is the optimal window for the new lattice.
Ultimately, we want to carry out these steps iteratively, as described below.

\begin{enumerate}
\item Choose an initial Gabor lattice,
\item Find the optimal window $g$,
\item Compute the lattice adapted to $g$,
\item Set this lattice as the Gabor lattice and loop to 2.
\end{enumerate}
Connected to the above procedure there are some rather delicate convergence questions.
Even though we will not provide theoretical justifications for doing this iterative loop of window and lattice optimization, we are able to back it up with numerical experiments.
Some of the experiments carried out numerically are described and visualized in Section \ref{apps}.

\section{Applications}\label{apps}

In this section we illustrate the method(s) with a few numerical examples. We focus on the main properties and results of the optimization procedures. A comprehensive description of their outputs for applications is not within the scope of the present study. A detailed analysis of their applications for audio signal analysis is planned in the near future. 

All the numerical applications have been computed using MATLAB and the LTFAT toolbox~\cite{ltfatnote015}. This toolbox contains functions for time-frequency analysis and synthesis. In particular we have used the discrete Gabor transform and functions designed for the construction of Gabor multipliers. It also contains a Gabor transform for non-separable (sheared) lattices as well as a function for displaying the spectrogram for these non conventional lattices.
\subsection{Some examples of optimal windows for the standard lattice}

\noindent{\bf Non-parametrized vs. parametrized.} We start by comparing the optimal windows obtained by minimizing the non-parametric case~\eqref{optprob} and the parametric case~\eqref{optprobparam}. The signal is a real valued quadratic chirp i.e. a bent pattern in the positive frequencies part of the TF plane and its symmetric counterpart in the negative ones. As a region of interest we pick the positive frequencies in a time window between $0.4$ and $0.7$ seconds. The results are plotted in Figure~\ref{nonparamvsparam}. The two optimal windows are significantly different. In the non-parametric case the window is allowed to have non-local or disconnected components. This freedom permits a better optimal window in terms of energy concentration in the TF plane. However, this window is not appropriate in terms of interpretation of the spectrogram as it introduces artificial components: a sort of shadow chirp component appears (see the spectrograms of the full signal on Fig.~\ref{nonparamvsparam}). The parametrized Gaussian is local and does not allow for these kind of "artefacts". Even if the optimization procedure giving this latter window is restricted to a small subspace, the improvement on the signal representation is noticeable. It is even comparable to the one computed with the non-parametric optimal window. The region chosen for the optimization is a thin line in the TF plane in both cases. Let us notice that the windows are optimal for a selected part of signal, which is the reason for the badly resolved chirps outside of the user selected region of interest.

\begin{figure}
\includegraphics[width=.49\linewidth]{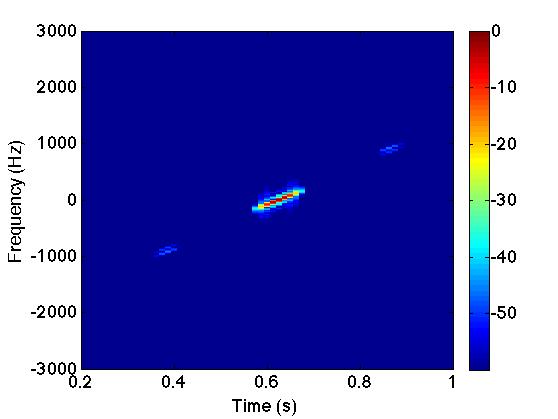}
\includegraphics[width=.49\linewidth]{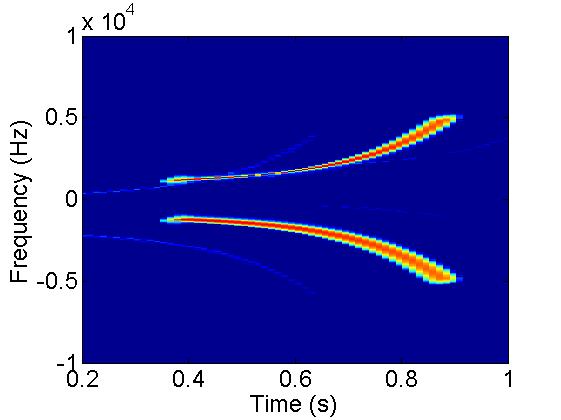}\\
\includegraphics[width=.49\linewidth]{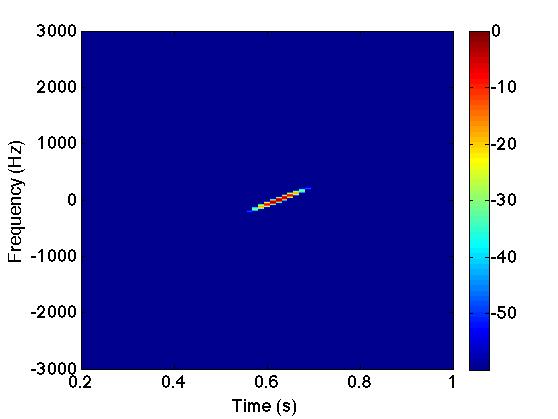}
\includegraphics[width=.49\linewidth]{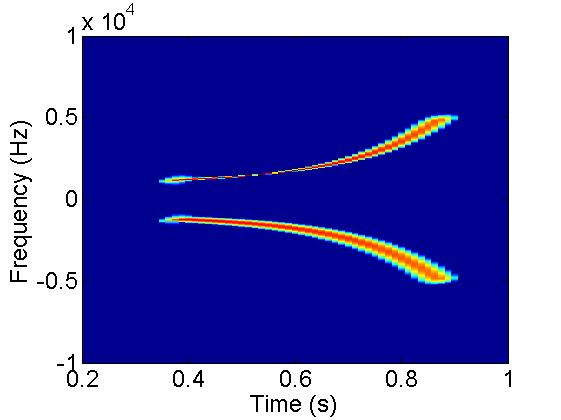}
\caption{Left: Modulus of  the optimal windows ambiguity function (top: non-parametric case, bottom: parametric case), log scale. Right: Spectrograms of a synthetic sound composed of a quadratic chirp using the two respective optimal windows on the left. The optimization has been done with respect to the part of the signal between time 0.4s and 0.7s and frequencies 0 to 5000 Hertz.}\label{nonparamvsparam}
\end{figure}

Let us remark that several tests were performed using large values of $p$ in the non-parametric case. As expected, the optimal window tends to match the extracted signal when $p$ increases (see remark~\ref{reminfinity}).

\noindent{\bf Accurate representation of a localized part.} The next test reveals the main asset of the method for time-frequency analysis. A better concentration in the TF plane implies a better separation of close components. This is illustrated on Fig.~\ref{TFsinus}. On the upper left image the spectrogram of a function $f$ has been computed with a standard (Gaussian) window. One can see that the time-frequency oscillating pattern is made of several nearby components, which seem to merge repeatedly over time. In fact the signal is composed of a sum of three very close frequency components with constant amplitude. These components have instantaneous frequencies $f_i$, with $i=1,2,3$, of the type:
\begin{equation}
f_i=\cos(b\cos c t+a_it+\theta),
\end{equation}
for some constants $a_i,b,c,\theta$.
The modulation is the same for all components so that they stay equally separated over time. The Gaussian window is not adapted for discriminating the three components around regions where the frequency is highly varying as for example around time $1.8$ seconds. As a consequence, the components appear as one thick line in the TF representation. Our window optimization procedure allows to find an appropriate window for this part of the spectrogram. We have extracted the part of the signal between 1.6 and 2s and have run the optimal window procedure. The result is shown on the upper right figure. This time the 3 components are separated around the region $t=1.8s$. However, the optimal window is of course optimal for only a part of the signal. One can see that this window is unable to separate the components well in the area where the Gaussian window succeeded. The bottom left figure could be the ideal representation for this type of signal. The signal has been cut in different successive pieces for which different optimal windows have been computed. For each part the Gabor transform has been computed with its optimal window. The total spectrogram is the sum of the spectrograms of the different parts. The bottom right figure shows the evolution of the energy maximum over time:
$$m(t)=\max_{\xi}{\cal V}_gf(\xi,t),
$$
for the different windows $g$. Since the windows are normalized, a higher maximum means a sparser representation.
\begin{figure}
\includegraphics[width=.49\linewidth]{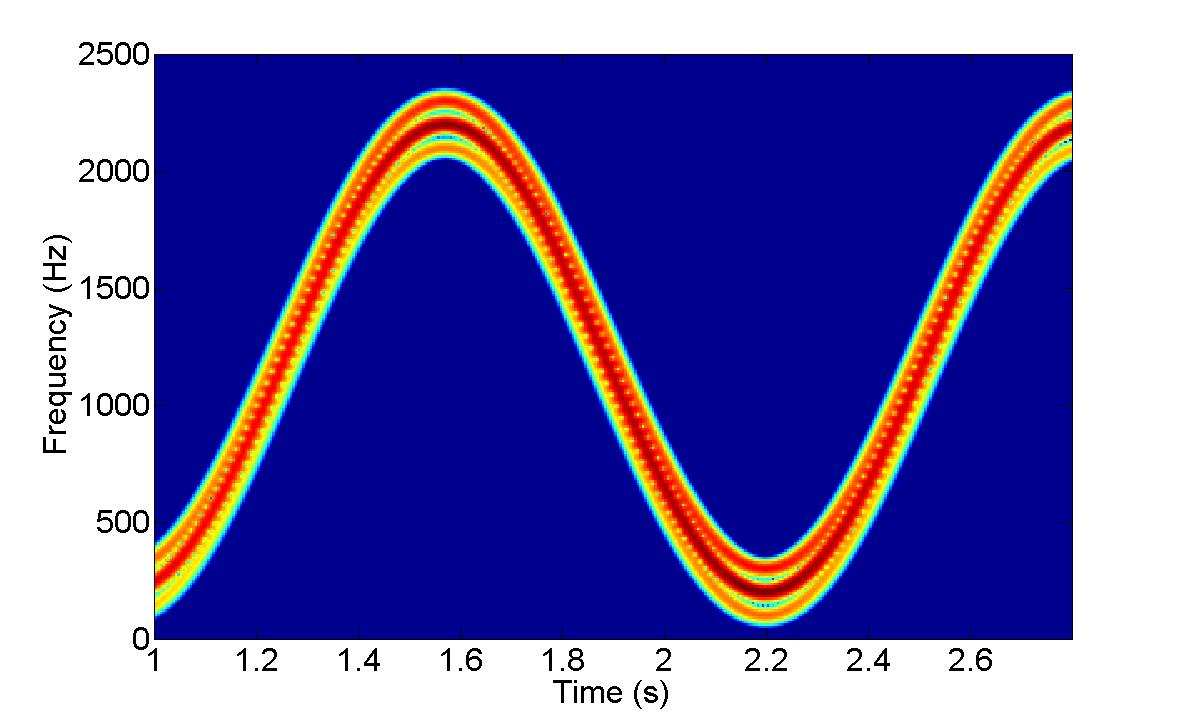}
\includegraphics[width=.49\linewidth]{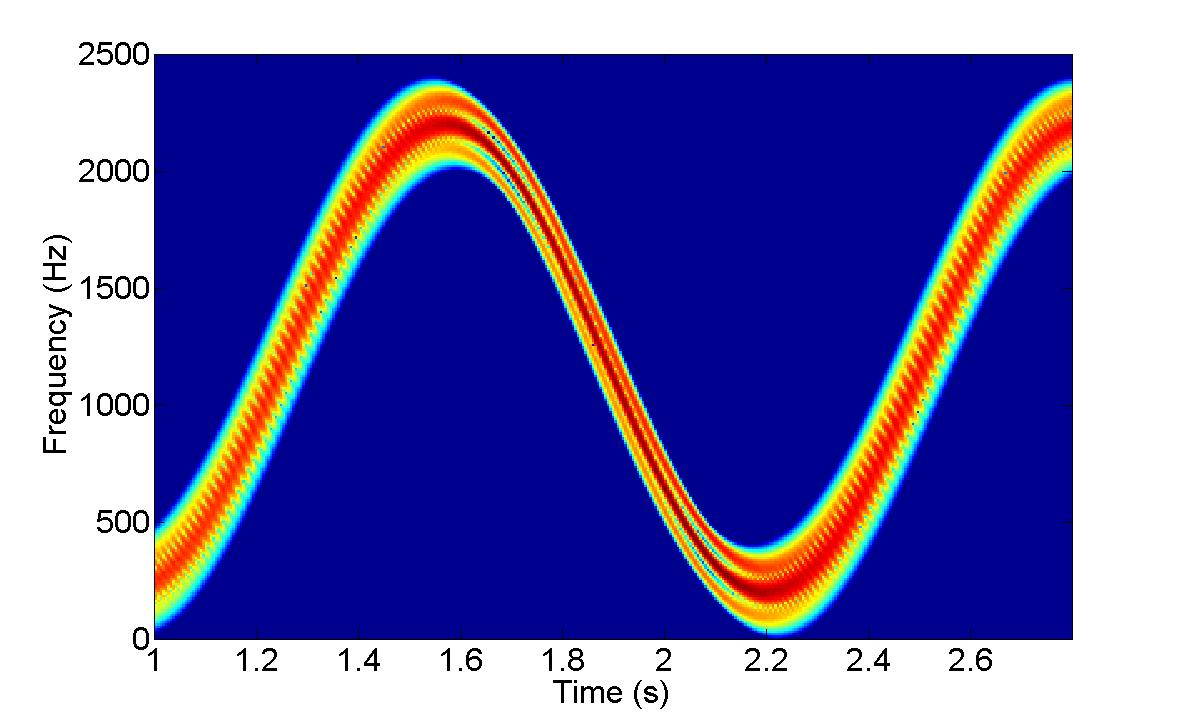}\\
\includegraphics[width=.49\linewidth]{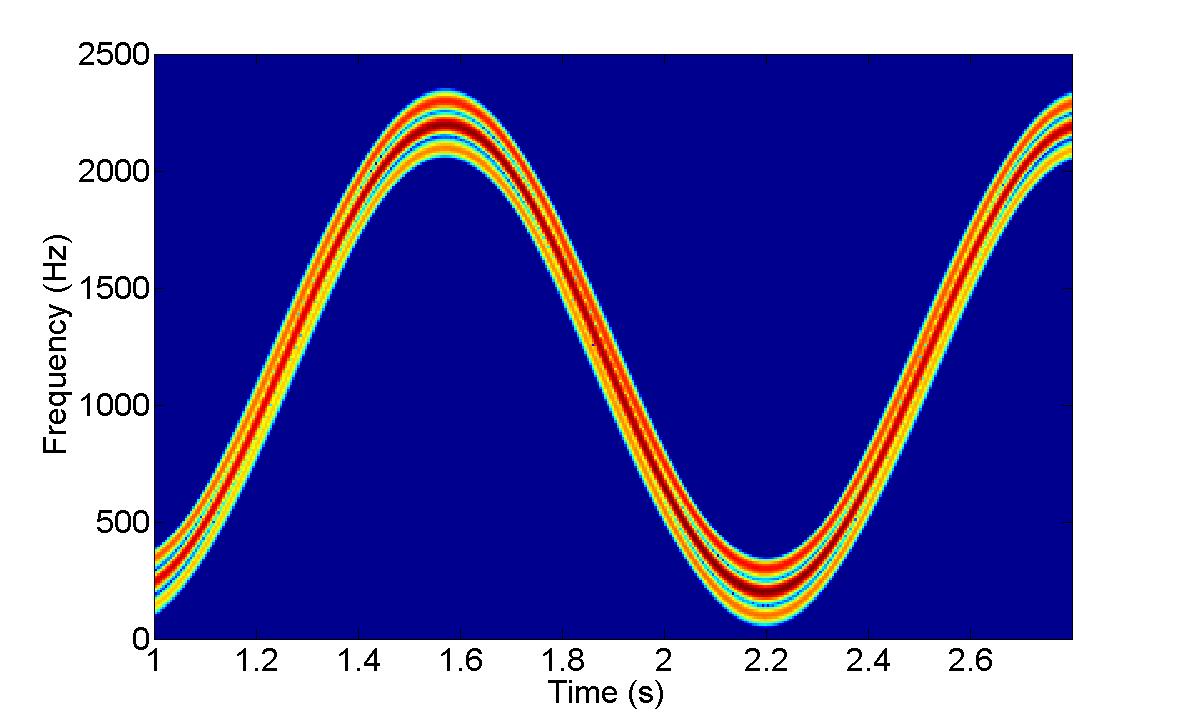}
\includegraphics[width=.49\linewidth]{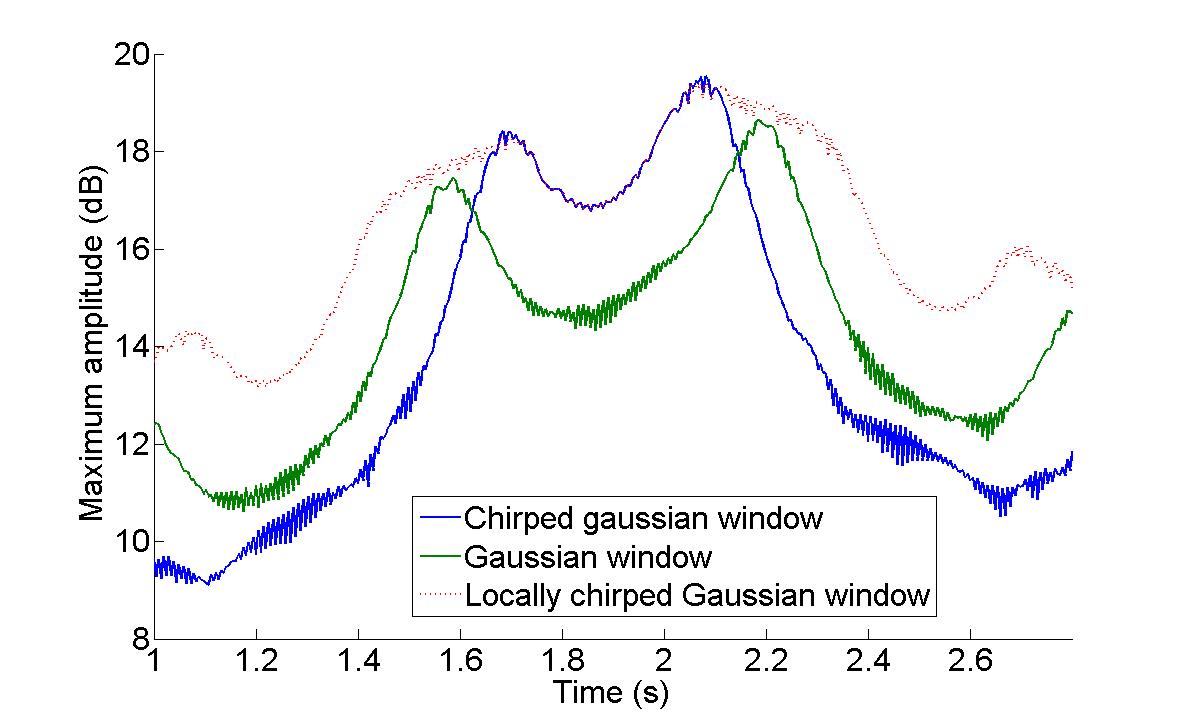}
\caption{Spectrogram of an oscillating chirp in the TF plane with three different windows. Top left: Gaussian window. Top right: Optimal window for time interval $[1.6,2]$. Bottom left: Optimal window evolving with time. Bottom right: Evolution in time of the maximum over frequencies of the spectrogram coefficients, for the 3 windows.}\label{TFsinus}
\end{figure}

\subsection{Combining optimal window and optimal lattice}~\label{optopt}

Given the optimal window, the computation of the adapted lattice gives the most appropriate time and frequency steps. These are two key parameters for an accurate representation. Fig~\ref{compsteps} illustrates this point. Even if the window is optimal, an inappropriate lattice will destroy the accuracy of the representation.

 \begin{figure}
 \includegraphics[width=.49\linewidth]{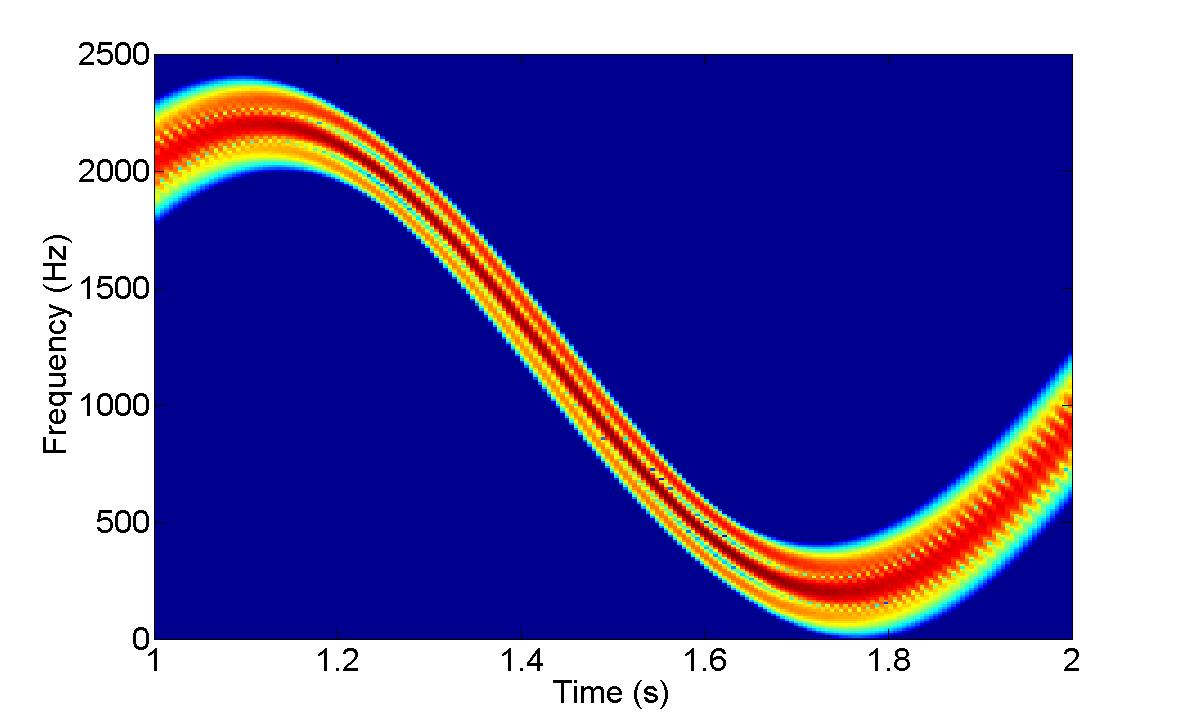}
\includegraphics[width=.49\linewidth]{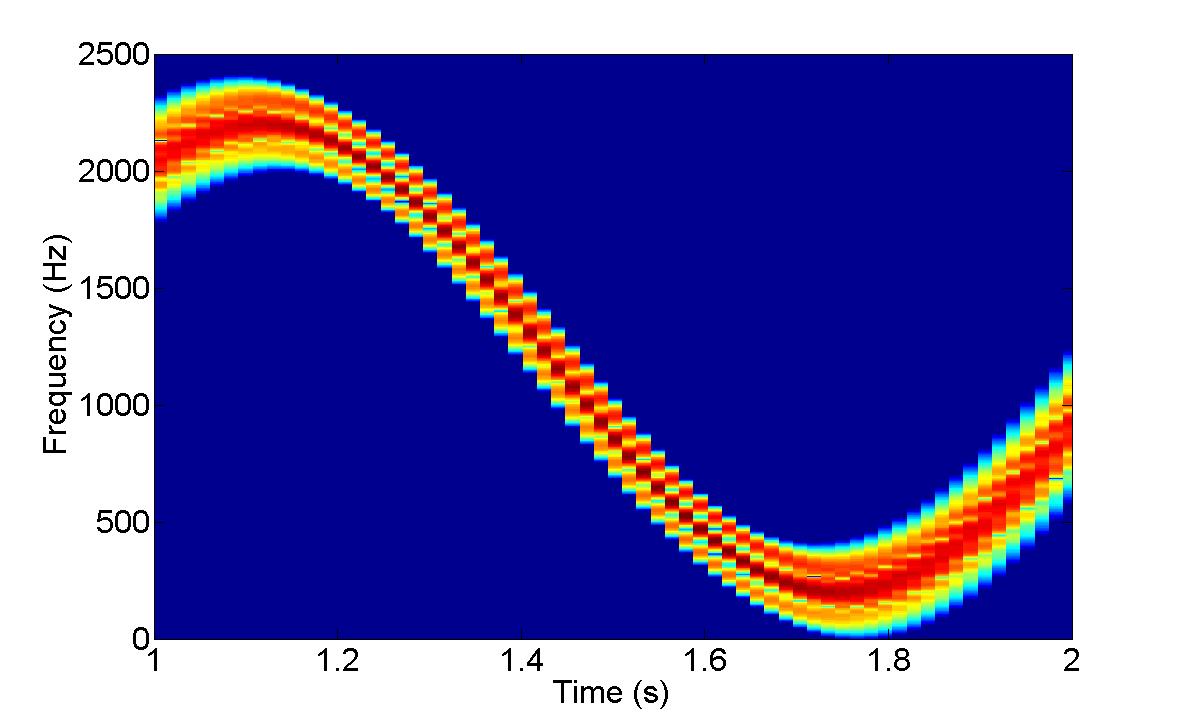}
\caption{Spectrograms of a signal with the same optimal window but for two different lattices having the same redundancy $R$. Left: Optimal time and frequency steps. Right: inappropriate lattice.}\label{compsteps}
\end{figure}

The whole optimization process containing the alternative window and lattice optimization (Sec.~\ref{Loopoptim}) has been tested in different cases. The preliminary results  show that the algorithm converges in many cases. However it appears that the redundancy parameter has an influence on the convergence. Choosing $R=15$ makes the algorithm converge rapidly in 4-5 iterations. 
For lower redundancies, the risk of alternating infinitely between several optimal windows and lattices increases.

Tests show that there are often only minor improvements of the representation sparsity when comparing representations using the optimal window/lattice at convergence and the optimal window/lattice given after the first iteration, see Fig.~\ref{compoptimall}. On this figure, the spectrograms using these two optimal windows are almost identical. One difference is in the position of the Gabor coefficients (see the zoom boxes); they are shifted along the frequency axis by a value varying at each time step, according to the shear parameter. In this case, the sparsity of the representation is not improved by iterating the optimal window/lattice procedure. However, the gain of using the optimal lattice resides in the denser packing in the TF plane of the analysis window. This is an important point for the synthesis aspect. By giving an optimal window and lattice, the optimization process of Sec.~\ref{Loopoptim} may further improve the packing, this is to be tested in future work.
Eventually, one must also notice that the TF pattern on which the optimization is made is rather simple: a short-time chirp signal. The results might be different for other types of patterns (to be investigated).
\begin{figure}
\includegraphics[width=.49\linewidth]{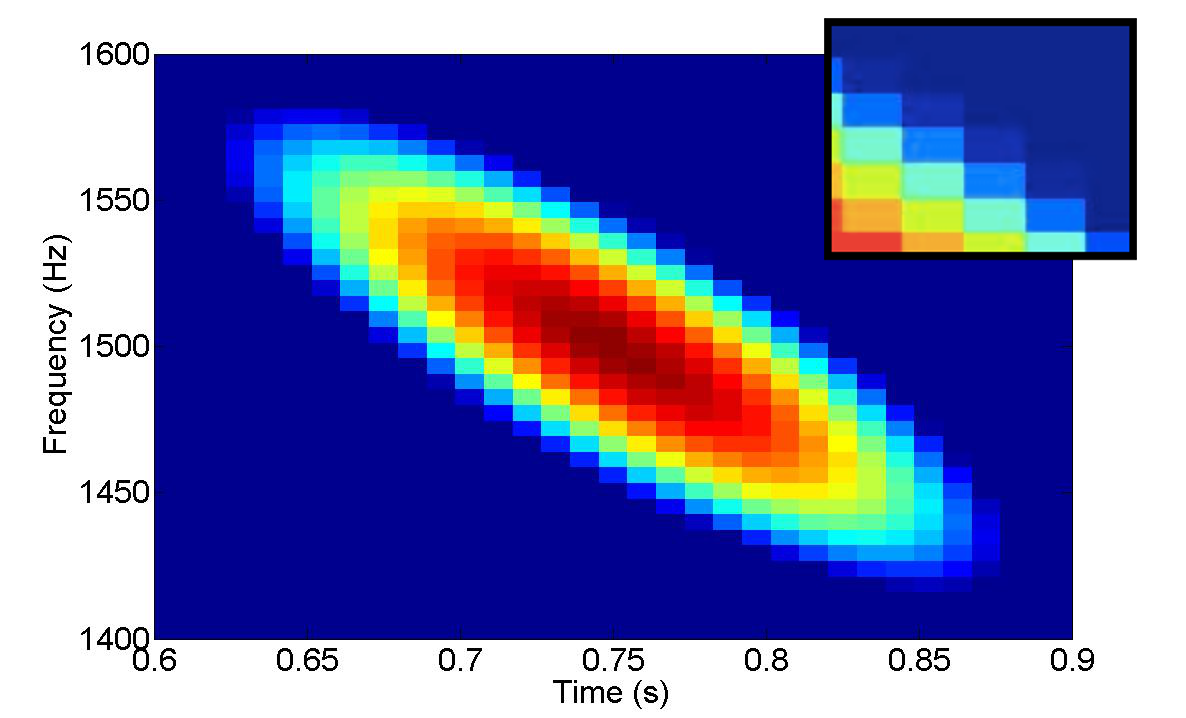}
\includegraphics[width=.49\linewidth]{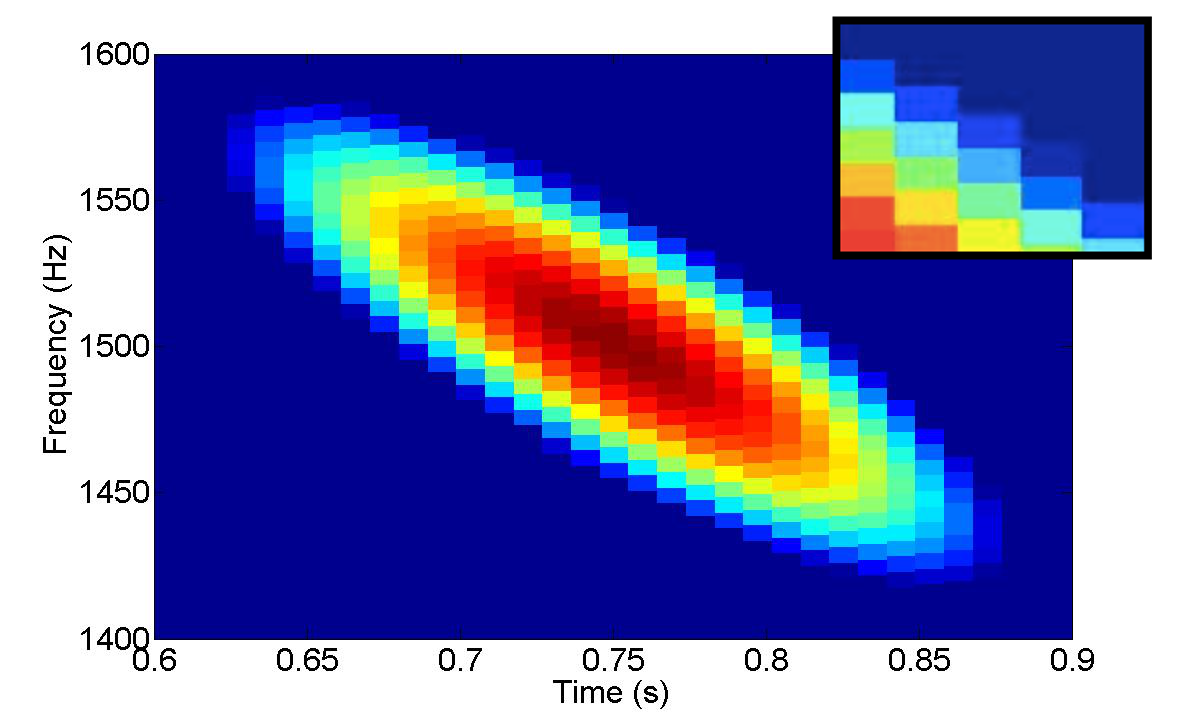}
\caption{Spectrograms of a chirp signal with two different windows and lattices. Left: with the optimal window associated to the optimal lattice obtained after the first iteration of the algorithm of Sec.~\ref{Loopoptim}. Right: with the optimal window and lattice after convergence of the algorithm of Sec.~\ref{Loopoptim}. In the top right corner of each picture a box magnifies the details of the TF tiling. Each pixel (rectangle) corresponds to the magnitude of one Gabor coefficient. The pixels are not aligned in the same manner (the shear values are different).}\label{compoptimall}
\end{figure}

\section{Conclusion and perspectives}

The work presented here shows promising results and several directions for further research.
The 2-parameter optimization scheme, although restricted, proved to be the best choice for the analysis of real world signals.
We noticed that it always converged to an appropriate optimal window for the patterns we were interested in.
We must keep in mind that this might not always be the case; it could reach a bad local optimum for unusual and highly complex patterns in the TF plane.
Another advantage over non-constraint problem is that the optimization of only two parameters  allows a fast computation.

The optimal lattice procedure is found to be highly valuable
as it adds extra power to the optimization both for the visualization (adapted lattice steps) and the synthesis (improving the condition number). It appears to be indispensable. In addition we provide a way for the fast computation of this adapted lattice (even though it is an approximation of it). This is an asset for practical applications.

It is particularly interesting to notice the relation between dilation and chirping of a window and dilation and shearing of a lattice. As seen in Sec.~\ref{optimallattice}, all possible Gabor lattices can be described by 3 parameters: dilation, shearing and redundancy. One can associate to any of these lattices a window parametrized by a chirp and a dilation which will be as optimally packed as would be round Gaussians with an hexagonal lattice. This points out the fact that the set of dilated and chirped version of any well-concentrated window is the natural and most general set of windows for the framework of the standard Gabor transform.

We demonstrated that the alternative lattice/window optimization can be done. This is only a preliminary work which serves as a proof of concept. We have seen that this principle is interesting for the automatic analysis of engineering signals. Future work on this topic must include a study of the convergence of the algorithm as well as an analysis of its relevance beyond a few iterations.

From a practical point of view, we found highly interesting results using a window whose shape evolves in time (c.f. Fig.~\ref{TFsinus}).
As a motivating example, it allows to follow and separate close frequency patterns over time.
The results shown on the figures were obtained by selecting manually successive time intervals and optimizing a window on each.
This process deserves to be made automatic.
However, a non-stationary window breaks the underlying structure of the Gabor transform.
The setting of non-stationary Gabor transforms \cite{badohojave11} is likely to be an appropriate setting to closer investigate windows evolving in time.
This gives a promising direction for future research.

 \bibliographystyle{abbrv}
 \bibliography{nhgbib,benjaminbib}

\end{document}